\documentclass[11pt]{amsart}

\usepackage{amsthm}
\usepackage{amsmath}
\usepackage{amssymb}
\usepackage{color}

\newcommand{\spn}{\mathop\mathrm{span}}

\newtheorem{thm}{Theorem}[section]
\newtheorem{theorem}[thm]{Theorem}

\newtheorem{corollary}[thm]{Corollary}

\newtheorem{lemma}[thm]{Lemma}
\newtheorem{proposition}[thm]{Proposition}

\newtheorem{definition}[thm]{Definition}
\theoremstyle{remark}
\newtheorem{remark}[thm]{Remark}

\newtheorem{example}[thm]{Example}

\begin{document}

\title{On phase and norm retrieval by subspaces}

\author[Tran, Huynh
]{Tin T. Tran and Phung T. Huynh\\
(In memory of Peter G. Casazza)}

\address{Tran: Department of Mathematics, University of Education, Hue University, Hue, Vietnam.}

\address{Huynh: Department of Mathematics, University of Sciences,
	Hue University, Hue, Vietnam.}


\email{tranthientin@hueuni.edu.vn}
\email{htphung@hueuni.edu.vn}

\begin{abstract} 
This paper studies phase and norm retrieval by subspaces. We first investigate norm retrieval by hyperplanes. We show that if $N$ hyperplanes $\{\varphi_i^\perp\}_{i=1}^N\subset \mathbb{R}^N$ allow norm retrieval and the vectors $\{\varphi_i\}_{i=1}^N$ are linearly independent, then these vectors must be an orthonormal basis for $\mathbb{R}^N$. We then present several new properties of subspaces that allow phase and norm retrieval. In particular, we provide a complete classification of two proper subspaces that perform norm retrieval. It is known that the collection of norm-retrievable frames $\{\varphi_i\}_{i=1}^M$ in $\mathbb{R}^N$ is not dense in the set of all $M$-element frames when $M < 2N-1$. We extend this result to subspaces. Several alternative proofs of fundamental results in phase and norm retrieval are also provided.

\vspace{0.2 in}

\end{abstract}
\maketitle

\section{Introduction and Preliminaries}

Hilbert space frames were introduced by Duffin and Schaeffer \cite{DS} in 1952 in connection with fundamental problems in non-harmonic Fourier series. For several decades, however, their ideas attracted limited attention outside this setting. The situation changed dramatically with the landmark paper of Daubechies et al. \cite{DGM} in 1986, after which the theory of frames gained widespread interest. Since then, frame theory has developed into a powerful tool with extensive applications across both pure and applied mathematics.

\begin{definition}
	Let $\mathcal{H}^N$ be a real or complex Hilbert space of dimension $N$. A sequence of vectors $\{\varphi_i\}_{i=1}^M$ is a frame for $\mathcal{H}^N$ if there exist constants $0<A\leq B<\infty$ such that
	\[A\|x\|^2\leq \sum_{i=1}^{M}|\langle x, \varphi_i\rangle|^2\leq B\|x\|^2, \ \text{ for all } x\in \mathcal{H}^N.\]
\end{definition}
The constants $A$ and $B$ are called the lower and upper frame bounds, respectively. If $A = B$ this is called an $A$-tight frame and if $A = B = 1$, it is referred to as a Parseval frame. Note that in a finite-dimensional Hilbert space, a frame is simply a spanning set for the space. For an introduction to frame theory, we refer the reader to \cite{CK, C}.

Signal reconstruction is a fundamental problem in engineering with a wide range of applications. One of the major challenges arises when signals must be recovered despite partial loss of information. In particular, the loss of phase information occurs in areas such as speech recognition \cite{BR, PDH, RJ} and optical applications such as X-ray crystallography and electron microscopy \cite{BM, F, F1}, making efficient phase retrieval essential. The notion of phase retrieval in the context of Hilbert space frames was first introduced in 2006 by Balan, Casazza, and Edidin \cite{BCE}, and since then it has developed into an active area of research; see, for example, \cite{A, BCC, BT, CCPW, E, H}, among others.

Phase retrieval has been studied for both vectors and projections and, in general, refers to the recovery of a signal’s phase from intensity measurements obtained via a redundant linear system. Another closely related area of research is norm retrieval, first introduced by Bahmanpour et al. \cite{BCC}, which concerns determining the norm of a vector from the absolute values of its intensity measurements. This notion arises naturally in the context of phase retrieval, particularly when considering collections of subspaces together with their orthogonal complements.

\begin{definition}
	A family of vectors $\{\varphi_i\}_{i=1}^M$ in $\mathcal{H}^N$ is said to do \emph{phase retrieval} (respectively, \emph{norm retrieval}) if for any $x, y \in \mathcal{H}^N$, the condition 
	\[
	|\langle x, \varphi_i \rangle| = |\langle y, \varphi_i \rangle| \quad \text{for all } i \in [M]
	\]
	implies that $x = cy$ for some scalar $c$ with $|c| = 1$ (respectively, $\|x\| = \|y\|$).
\end{definition}
In this definition and throughout this paper, for any $M \in \mathbb{N}$, we denote by $[M]$ the set $\{1, 2, \ldots, M\}$.

It is immediate that phase retrieval implies norm retrieval. Every orthonormal basis performs norm retrieval but fails to achieve phase retrieval. More generally, any set of vectors containing a tight frame guarantees norm retrieval.

In the real case, the complement property provides a fundamental characterization of phase retrieval, as established in \cite{BCE}.
\begin{definition}
	A family of vectors $\{\varphi_i\}_{i=1}^M \subset \mathbb{R}^N$ is said to have the \emph{complement property} if for every index set $I \subset [M]$, either 
	\[
	\spn\{\varphi_i : i \in I\} = \mathbb{R}^N 
	\quad \text{or} \quad 
	\spn\{\varphi_i : i \in I^c\} = \mathbb{R}^N.
	\]
\end{definition}
\begin{theorem}[\cite{BCE}]\label{com}
	A family of vectors $\{\varphi_i\}_{i=1}^M$ in $\mathbb{R}^N$ does phase retrieval if and only if it has the complement property.
\end{theorem}

It follows that if a set of vectors $\{\varphi_i\}_{i=1}^M$ does phase retrieval in $\mathbb{R}^N$, then $M\geq 2N-1$. 
Note that in the complex case, the complement property is a necessary (but not sufficient) condition for phase retrieval, see \cite{BCE}.

Full spark is another important notion of vectors in frame theory. 
\begin{definition}
	Given a family of vectors $\Phi=\{\varphi_i\}_{i=1}^M$ in $\mathcal{H}^N$, the spark of $\Phi$
	is defined as the cardinality of the smallest linearly dependent subset of $\Phi$. When spark$(\Phi) = N + 1$, every subset of size $N$ is linearly independent, and in that case,
	$\Phi$ is said to be full spark.
\end{definition}
From the definitions, it follows that full spark frames with $M \geq
2N - 1$ vectors have the complement property and hence do phase retrieval. 
Also, if $M = 2N - 1$ then the complement property clearly implies full spark.

In some cases, such as crystal twinning in X-ray crystallography \cite{D}, signals must be reconstructed from the norms of higher-dimensional components. This led to the development of phase retrieval by projections.

\begin{definition}
	Let $\{W_i\}_{i=1}^M$ be a family of subspaces of $\mathcal{H}^N$ with corresponding orthogonal projections $\{P_i\}_{i=1}^M$. We say that $\{W_i\}_{i=1}^M$ (or $\{P_i\}_{i=1}^M$) yields \emph{phase retrieval} (respectively, \emph{norm retrieval}) if for any $x, y \in \mathcal{H}^N$ satisfying
	\[
	\|P_i x\| = \|P_i y\| \quad \text{for all } i \in [M],
	\]
	then $x = c y$ for some scalar $c$ with $|c| = 1$ (respectively, $\|x\| = \|y\|$).
\end{definition}

The following is a fundamental result in phase retrieval by projections.

\begin{theorem}[Edidin, \cite{E}]
	A family of orthogonal projections $\{P_i\}_{i=1}^M$ does phase retrieval in $\mathbb{R}^N$ if and only if for every non-zero vector $x\in \mathbb{R}^N$, $\spn\{P_ix\}_{i=1}^M=\mathbb{R}^N$. 
\end{theorem}

Given subspaces $\{W_i\}_{i=1}^M$ of $\mathcal{H}^N$ which yield phase retrieval, it is not necessarily the case that the family of orthogonal complements $\{W_i^\perp\}_{i=1}^M$ also yields phase retrieval. The following result shows that norm retrieval is precisely the condition required for phase retrieval to pass to orthogonal complements.
\begin{theorem}[\cite{BCC}]\label{passorth}
	Let $\{W_i\}_{i=1}^M$ be subspaces of $\mathcal{H}^N$ which do phase retrieval. Then $\{W_i^\perp\}_{i=1}^M$ does phase retrieval if and only if it does norm retrieval.
\end{theorem}

In this paper, we continue the study of phase and norm retrieval by subspaces. The remainder of the paper is organized as follows. In Section 2, we investigate norm retrieval by hyperplanes. In Section 3, we present several new properties of subspaces that perform phase and norm retrieval. It is known that the collection of norm-retrievable frames $\{\varphi_i\}_{i=1}^M$ in $\mathbb{R}^N$ is not dense in the set of all $M$-element frames when $M < 2N-1$. We extend this result to the setting of subspaces. Several alternative proofs of fundamental results in phase and norm retrieval are also provided in this section.

\section{Properties of hyperplanes doing norm retrieval}

In this section, we are interested in sets of vectors whose orthogonal complements perform norm retrieval. It is clear that if a set contains an orthonormal basis, then both the set and its orthogonal complements do norm retrieval. The following theorem appears in \cite{Ca}.

\begin{theorem}[\cite{Ca}] \label{thm_Ca1}
	If the set of vectors $\{\varphi_i\}_{i=1}^{N-1}\subset\mathbb{R}^N$ is linearly independent, then its orthogonal complement $\{\varphi_i^\perp\}_{i=1}^{N-1}$ cannot do norm retrieval.
\end{theorem} 
The following is a direct consequence of Theorem \ref{thm_Ca1}.
\begin{corollary}
Let $\{\varphi_i\}_{i=1}^M\subset \mathbb{R}^N$ be a set of non-zero vectors. If the set of hyperplanes $\{\varphi_i^\perp\}_{i=1}^M$ does norm retrieval, then either $\{\varphi_i\}_{i=1}^M$ spans $\mathbb{R}^N$ or is linearly dependent.
\end{corollary}
We now consider the case where the independent set has exactly $N$ vectors. We first need the following lemma.
	\begin{lemma}\label{ortho}
	Suppose $\{\varphi_i\}_{i=1}^N$
	is a set of linearly independent unit vectors such that the hyperplanes $\{\varphi_i^{\perp}\}_{i=1}^N$ do norm retrieval in $\mathbb{R}^N$. If $x$ and $y$ are vectors in $\mathbb{R}^N$ satisfying 
	\[\langle x, \varphi_i\rangle = \pm1 \mbox{ and } \langle y, \varphi_i\rangle =\pm 1 \text{ for all } i \in [N],\] then $\|x\|=\|y\|$.
\end{lemma}
\begin{proof}
	Let $P_i$ be the orthogonal projection onto $\varphi_i^{\perp}, i\in [N]$. Then we have
	\[ x=P_1x\pm\varphi_1=P_2x\pm\varphi_2=\cdots = P_Nx\pm\varphi_N,\]
	and
	\[ y=P_1y\pm\varphi_1=P_2y\pm\varphi_2=\cdots= P_Ny\pm\varphi_N.\]
	Moreover,
	\[ \alpha^2:= \|P_1x\|^2=\|P_2x\|^2=\cdots=\|P_Nx\|^2=\|x\|^2-1,\]
	\[ \beta^2:=\|P_1y\|^2=\|P_2y\|^2=\cdots= \|P_Ny\|^2=\|y\|^2-1,\]
	where $\alpha,\beta>0$.
	
	Note that
	\[ \|P_i(x/\alpha)\|^2=\|P_i(y/\beta)\|^2=1\mbox{ for all }i\in[N].\]
	By our assumption of norm retrieval, $\|x/\alpha\|^2=\|y/\beta\|^2.$ Note also that
	
	\[ \|x/\alpha\|^2-1/\alpha^2 =1=\|y/\beta\|^2-1/\beta^2.\]
	Hence, $\alpha=\beta$ and so $\|x\|=\|y\|$.
\end{proof}

\begin{theorem}\label{thm_M}
	If the hyperplanes $\{\varphi_i^{\perp}\}_{i=1}^N$ do norm retrieval in $\mathbb{R}^N$ and $\{\varphi_i\}_{i=1}^N$
	is a linearly independent set of unit vectors, then $\{\varphi_i\}_{i=1}^N$ must be an orthonormal basis for $\mathbb{R}^N$.
\end{theorem}
\begin{proof}
	Define \[L=\{(\lambda_1, \lambda_2, \ldots, \lambda_N) : \lambda_i=\pm 1, i=1, 2, \ldots, N-1, \lambda_N=1\}.\] 
	For each $\lambda = (\lambda_1, \lambda_2, \ldots, \lambda_N)\in L$, let $x(\lambda)$ be the (unique) solution of the system of linear equations:
	\[\langle x, \varphi_i\rangle = \lambda_i, \quad i\in [N].\] By Lemma \ref{ortho} we have that $\|x(\lambda)\|=\|x(\lambda')\|$ for all $\lambda, \lambda'\in L$.
	
	Let $\overline{x}=\sum_{\lambda\in L}x(\lambda)$. We will show that $$\overline{x}\perp \varphi_i \mbox{ and } \overline{x}\in \spn\{\varphi_j : j\in [N], j\not=i\}$$ for all $i=1, 2, \ldots, N -1.$
	
	Indeed, without loss of generality, we can assume that $i=1$. Let 
	\[K=\{(\mu_2, \mu_3, \ldots, \mu_N): \mu_j=\pm 1, j=2, \ldots, N-1, \mu_N=1 \}.\] Note that $L=\{(1, \mu), (-1, \mu) : \mu=(\mu_2, \mu_3, \ldots, \mu_N)\in K\}$. For each $\mu\in K$, we have that
	\[\langle x(1, \mu), \varphi_1\rangle=1 \mbox{ and } \langle x(-1, \mu), \varphi_1\rangle=-1.\] Therefore, $x(1, \mu)+x(-1, \mu)\perp \varphi_1$ and hence, 
	\[\overline{x}=\sum_{\mu\in K}(x(1, \mu)+x(-1, \mu))\perp \varphi_1.\]
	Moreover, for each $\mu\in K$, we have that
	\[\langle x(1, \mu), \varphi_j\rangle=\langle x(-1, \mu), \varphi_j\rangle=\mu_j, \mbox{ for all } j=2, \ldots, N.\] This implies that 
	\[x(1, \mu)-x(-1, \mu)\perp \varphi_j, \mbox{ for all } j=2, \ldots, N.\]
	Since $\|x(1, \mu)\|=\|x(-1, \mu)\|$, it follows that $x(1, \mu)+x(-1, \mu)$ is orthogonal to $x(1, \mu)-x(-1, \mu)$ for all $\mu\in K$. Thus, $x(1, \mu)+x(-1, \mu)\in \spn\{\varphi_2, \ldots,\varphi_N\}$ for all $\mu\in K$ and so $$\overline{x}=\sum_{\mu\in K}(x(1, \mu)+x(-1, \mu))\in \spn \{\varphi_2, \ldots, \varphi_N\}.$$ This completes the proof of the claim.
	
	Since $\{\varphi_i\}_{i=1}^N$ is linearly independent, $\overline{x}$ has a unique representation as a linear combination of the vectors $\varphi_i$. Combining this with the fact that $ \overline{x}\in \spn\{\varphi_j : j\in [N], j\not=i\}$ for all $i=1, 2, \ldots N-1,$ we must have 
	$\overline{x}=\alpha \varphi_N$, for some scalar $\alpha$.
	
	Note that \[\alpha=\langle \overline{x}, \varphi_N\rangle=\sum_{\lambda\in L}\langle x(\lambda), \varphi_N\rangle=\sum_{\lambda\in L}1=2^{N-1}\not=0,\] and $\overline{x}\perp \varphi_i$ for $i=1, 2, \ldots, N-1$. So $\varphi_N$ is orthogonal to all other $\varphi_i$. This argument clearly iterates
	to show that $\varphi_i\perp \varphi_j$, for all $i\not= j$.
\end{proof}

It is worth mentioning that norm retrieval can be achieved using three hyperplanes in $\mathbb{R}^N$. This result was previously stated in \cite{Ca}; however, the proof provided there is unclear, as it does not demonstrate how $\|x\|$ can be reconstructed from the projection norms. In particular, the coefficients appearing in their linear combination are unknown. We will provide a correct and complete proof of this result.

\begin{theorem}
	In $\mathbb{R}^N$, three proper subspaces of codimension one can do norm retrieval.
\end{theorem}

\begin{proof}
	Let $\{e_i\}_{i=1}^N$ be the canonical orthonormal basis of $\mathbb{R}^N$, and define
	\[
	\varphi_1 = e_1, 
	\quad 
	\varphi_2 = e_2, 
	\quad 
	\varphi_3 = \frac{e_1 - e_2}{\sqrt{2}}.
	\]
	Let $P_i$ denote the orthogonal projection onto $\varphi_i^\perp$, $i=1,2,3$.
	We will show that the family $\{\varphi_i^\perp\}_{i=1}^3$ performs norm retrieval.
	
	Let $x=(a_1,a_2,\dots,a_N)\in\mathbb{R}^N$ and set
	\[
	S := \sum_{i=3}^N a_i^2.
	\]
	Then we have that
	\begin{align}
		\|P_1 x\|^2 &= a_2^2 + S, \label{eq:P1}
		\end{align}
		\begin{align}
		\|P_2 x\|^2 &= a_1^2 + S, \label{eq:P2}
		\end{align}
		\begin{align}
		\|P_3 x\|^2 &= \frac{(a_1+a_2)^2}{2} + S. \label{eq:P3}
	\end{align}
		
	Subtracting \eqref{eq:P1} from \eqref{eq:P2} gives
	\begin{equation}\label{eq:diff}
		\|P_2 x\|^2 - \|P_1 x\|^2 = a_1^2 - a_2^2.
	\end{equation}
	Moreover, from \eqref{eq:P1}, \eqref{eq:P2}, and \eqref{eq:P3} we have
	\begin{align*}
		\|P_1 x\|^2 + \|P_2 x\|^2 - 2\|P_3 x\|^2
		&= a_1^2 + a_2^2 - (a_1+a_2)^2 \\
		&= -2a_1a_2.
	\end{align*}
	Hence
	\begin{equation}\label{eq:prod}
		a_1a_2
		= \|P_3 x\|^2 - \frac12(\|P_1 x\|^2 + \|P_2 x\|^2).
	\end{equation}
	
	From \eqref{eq:diff} and \eqref{eq:prod}, both $a_1^2-a_2^2$ and $a_1a_2$ are determined by the projection norms. Consequently,
	\[
	(a_1^2+a_2^2)^2
	= (a_1^2-a_2^2)^2 + 4a_1^2a_2^2
	\]
	is determined, and since $a_1^2+a_2^2\ge 0$, the quantity $a_1^2+a_2^2$ is uniquely determined.
	
It follows that 
	\[
	\|x\|^2
	= a_1^2 + a_2^2 + S
	= \dfrac{1}{2}\left(\|P_1 x\|^2 + \|P_2 x\|^2\right) +\dfrac{1}{2}\left(a_1^2+a_2^2\right),
	\]
	which depends only on the projection norms.
	Therefore, the hyperplanes $\{\varphi_i^\perp\}_{i=1}^3$ perform norm retrieval in $\mathbb{R}^N$.
\end{proof}

\begin{remark}
	Three hyperplanes: $e_1^\perp, e_2^\perp$ and $\left((e_1+e_2)/\sqrt{2}\right)^\perp$ also do norm retrieval.
\end{remark}
The following provides a construction in which both the vectors and their orthogonal complements yield phase retrieval.
\begin{theorem}
	Let $\{e_i\}_{i=1}^N$ and $\{\psi_i\}_{i=1}^N$ be orthonormal bases for $\mathbb{R}^N$ such that, when taken together, they form a full spark family. Define
	$\varphi_i = e_i \text{ for } i=1,2,\ldots,N, 
	\text{ and } 
	\varphi_{N+i} = \psi_i  \text{ for } i=1,2,\ldots,N-1.$
	Then both $\{\varphi_i\}_{i=1}^{2N-1}$ and $\{\varphi_i^{\perp}\}_{i=1}^{2N-1}$ yield phase retrieval.
\end{theorem}
\begin{proof}
	If we partition the family $\{\varphi_i\}_{i=1}^{2N-1}$ into two sets, then one of them must contain at least $N$ elements. By the full spark property, this subset spans $\mathbb{R}^N$. In other words, this family satisfies the complement property and therefore yields phase retrieval.
	
	Let $P_i$ be the orthogonal projection onto $\varphi_i^\perp$, for $i\in [N]$. Note that for all $x\in\mathbb{R}^N$, we have
	\[
	\|P_i x\|^2 = \sum_{j \neq i} |\langle x, e_j\rangle|^2,\] 
	and hence
	\[\sum_{i=1}^N \|P_i x\|^2 = (N-1)\|x\|^2.
	\]
	It follows that $\{\varphi_i^{\perp}\}_{i=1}^{2N-1}$ yields norm retrieval, and consequently phase retrieval by Theorem~\ref{passorth}.
\end{proof}
\section{Phase and Norm retrieval by subspaces}
In this section, we present new results on phase and norm retrieval by subspaces. We first provide an alternative proof of an important classification result for subspaces that perform phase retrieval. The following lemma is the key component of the proof and is applicable to both real and complex settings.

\begin{lemma}\label{L4} Let $W$ be an $p$-dimensional subspace of $\mathcal{H}^N$. If $x,y\in W$ satisfy $\|x\|=\|y\|$, then there exists an orthonormal basis $\{\varphi_i\}_{i=1}^p$ for $W$ such that $|\langle x, \varphi_i \rangle| = |\langle y, \varphi_i \rangle|$ for all $i \in [p]$.
\end{lemma}
\begin{proof}
	If $x=y$ or $p=1$, the statement is trivial. Hence, we assume that $x\neq y$ and $p\geq 2$.
	
	Since $p\geq 2$, we can choose a unit vector $\varphi_1\in W \cap (x-y)^\perp$. Then
	\[
	\langle x,\varphi_1\rangle=\langle y,\varphi_1\rangle.
	\]
	If $p>2$, we choose a unit vector $\varphi_2\in W \cap \{\varphi_1, x-y\}^\perp$. Then
	\[\varphi_2\perp \varphi_1 \mbox{ and } \langle x, \varphi_2\rangle = \langle y, \varphi_2\rangle.\]
	
	Continuing this process, we obtain an orthonormal set $\{\varphi_i\}_{i=1}^{p-1}\subset W$ satisfying $\langle x, \varphi_i \rangle = \langle y, \varphi_i \rangle$ for all $i = 1, 2, \dots, p-1$.
	
Finally, choose a unit vector $\varphi_p$ in $W$ that is orthogonal to all the vectors $\{\varphi_i\}_{i=1}^{p-1}$. Then $\{\varphi_i\}_{i=1}^p$ is an orthonormal basis for $W$. Moreover, 
	\[	\sum_{i=1}^p |\langle x, \varphi_i \rangle|^2 = \|x\|^2 = \|y\|^2 = \sum_{i=1}^p |\langle y, \varphi_i \rangle|^2.\]
	It follows that $|\langle x, \varphi_p\rangle|=|\langle y, \varphi_p\rangle|$.
	This completes the proof.
\end{proof}
\begin{theorem}[\cite{CCPW}, \cite{CCJW}]\label{T21}
	Let $\{W_i\}_{i=1}^M$ be subspaces of $\mathcal{H}^N$ with corresponding orthogonal projections $\{P_i\}_{i=1}^M$. The following are equivalent:
	\begin{enumerate}
		\item $\{W_i\}_{i=1}^M$ does phase retrieval.
		\item If $\{\varphi_{ij}\}_{j=1}^{N_i}$ is an orthonormal basis for $W_i$, $i\in [M]$, then the collection $\{\varphi_{ij}\}_{i=1,j=1}^{M,\ N_i}$
		does phase retrieval.
	\end{enumerate}
\end{theorem}
\begin{proof}
	$(1)\Rightarrow (2)$:
	Assume $x,y\in \mathcal{H}^N$ satisfy $|\langle x, \varphi_{ij}\rangle|=|\langle y, \varphi_{ij}\rangle$ for all $i,j$.
	Then for each $i\in [M]$, we have
	\[ \|P_ix\|^2=\sum_{j=1}^{N_i}|\langle x,\varphi_{ij}\rangle|^2=\sum_{j=1}^{N_i}|\langle y,\varphi_{ij}\rangle|^2=\|P_iy\|^2.\]
	By (1), it follows that $x=cy$ for some scalar $c$ with $|c|=1$.
	\vskip7pt
	
	$(2)\Rightarrow (1)$:
	Let $x,y\in \mathcal{H}^N$ satisfy $\|P_ix\|=\|P_iy\|$ for all $i\in [M]$. By Lemma \ref{L4}, there exist orthonormal
	bases $\{\varphi_{ij}\}_{j=1}^{N_i}$ for $W_i$ such that $|\langle x,\varphi_{ij}\rangle|=|\langle y,\varphi_{ij}\rangle |$
	for all $i, j$. Since (2) assumes the collection $\{\varphi_{ij}\}_{i=1,j=1}^{M,\ N_i}$ does phase retrieval, we conclude that $x=cy$ with $|c|=1$.
\end{proof}
\begin{remark}
	In \cite{CCPW}, the authors provide a proof of Theorem \ref{T21} for the real case, while the proof for the complex case appears in \cite{CCJW}. However, these proofs are rather technical, particularly in the complex setting.
\end{remark}
The following theorem appears in \cite{CC}, but we will provide a proof here for completeness.

\begin{theorem}\label{thm_PD}
	Given orthogonal projections $\{P_i\}_{i=1}^M$ on $\mathbb{R}^N$, the following statements are equivalent.
	\begin{enumerate}
		\item $\{P_i\}_{i=1}^M$ does norm retrieval.
		\item For every $x\in\mathbb{R}^N$, we have $ [\spn\{P_ix\}_{i=1}^M]^\perp\subset x^\perp$.
		\item For every $x\in\mathbb{R}^N$, we have $x\in \spn\{P_ix\}_{i=1}^M$.
	\end{enumerate}
\end{theorem}
\begin{proof}
	$(1)\Rightarrow (2)$: Let $y\in [\spn\{P_ix\}_{i=1}^M]^\perp$. Then 
	\[\langle P_iy, P_ix\rangle=\langle y, P_ix\rangle= 0, \ \mbox{ for all } i\in [M].\]
	Let $u=x+y, v=x-y$, then 
	\[\|P_iu\|^2=\langle P_ix+P_iy, P_ix+P_iy\rangle=\|P_ix\|^2+\|P_iy\|^2,\]
	\[\|P_iv\|^2=\langle P_ix-P_iy, P_ix-P_iy\rangle=\|P_ix\|^2+\|P_iy\|^2,\] for all $i$. Therefore, $\|u\|=\|v\|$ and so $y\perp x$.
	
	$(2)\Rightarrow (1)$: Suppose that $x, y\in\mathbb{R}^n$ satisfy $\|P_ix\|=\|P_iy\|$ for all $i$. Let $u=x+y, v=x-y$, then 
	\[\langle u, P_iv\rangle=\langle P_iu, P_iv\rangle=\langle P_ix+P_iy, P_ix-P_iy\rangle=0, \ \mbox{ for all } i\in [M].\]
	Thus, $u\perp \spn\{P_iv\}_{i=1}^M$. Since $[\spn\{P_iv\}_{i=1}^M]^\perp\subset v^\perp$, it follows that $u\perp v$ and hence $\|x\|=\|y\|$.
	
	$(2)\Leftrightarrow (3)$: This follows immediately.
\end{proof}

In \cite{H}, the author uses the spanning property of frame elements to provide an important classification of norm-retrievable frames in $\mathbb{R}^N$. We present an alternative proof of this result as a consequence of Theorem \ref{thm_PD}.

\begin{theorem}\label{Fa}
	A frame $\{\varphi_i\}_{i=1}^M$ for $\mathbb{R}^N$ does norm retrieval if any only if for any partition $\{I_1, I_2\}$ of $[M]$, $[\spn\{\varphi_i\}_{i\in I_1}]^\perp\perp$ $[\spn\{\varphi_i\}_{i\in I_2}]^\perp$.
\end{theorem}

\begin{proof} Let $P_i$ be the orthogonal projection onto $\spn\{\varphi_i\}$ for $i\in [M]$. Suppose that $\{\varphi_i\}_{i=1}^M$ does norm retrieval in $\mathbb{R}^N$. Let $\{I_1, I_2\}$ be an arbitrary partition of $[M]$. For any $x\in [\spn\{\varphi_i\}_{i\in I_1}]^\perp$, Theorem \ref{thm_PD} gives
	\[x\in \spn\{P_ix\}_{i=1}^M=\spn\{P_ix\}_{i\in I_2}\subset \spn\{\varphi_i\}_{i\in I_2}.\] Thus, we conclude that $[\spn\{\varphi_i\}_{i\in I_1}]^\perp\subset \spn\{\varphi_i\}_{i\in I_2}$, completing the proof in one direction.
	
	Conversely, suppose that $[\spn\{\varphi_i\}_{i\in I_1}]^\perp\perp[\spn\{\varphi_i\}_{i\in I_2}]^\perp$ for any partition $\{I_1, I_2\}$ of $[M]$. Let $x\in \mathbb{R}^N$ and $y\in [\spn\{P_ix\}_{i=1}^M]^\perp$. Define the index set
	\[I:=\{i\in [M]: P_ix=0\}.\] Then $x \perp \varphi_i$ for all $i\in I$ and $y\in[\spn\{P_ix\}_{i\in I^c}]^\perp=[\spn\{\varphi_i\}_{i\in I^c}]^\perp.$ By assumption, $$[\spn\{\varphi_i\}_{i\in I}]^\perp\perp[\spn\{\varphi_i\}_{i\in I^c}]^\perp,$$ which implies $y\perp x$. By Theorem \ref{thm_PD}, $\{\varphi_i\}_{i=1}^M$ does norm retrieval.
\end{proof}

\begin{proposition}\label{Pro_span_ortho} Suppose that $\{\varphi_i\}_{i=1}^M$ does norm retrieval in $\mathbb{R}^N$. Then for every $j\in [M]$, if $\varphi_j\notin\spn_{i\not=j}\{\varphi_i\}$, then $\varphi_j$ is orthogonal to all other $\varphi_i$. Consequently, a basis does norm retrieval in $\mathbb{R}^N$ if and only if it is an orthogonal basis.
\end{proposition}
\begin{proof}
	Since $\{\varphi_i\}_{i=1}^M$ does norm retrieval, it spans $\mathbb{R}^N$. Therefore, if $\varphi_j\notin\spn_{i\not=j}\{\varphi_i\}$, then $\spn_{i\not=j}\{\varphi_i\}$ is a hyperplane in $\mathbb{R}^N$. By Theorem \ref{Fa}, this hyperplane is $\varphi_j^\perp$. The conclusion follows.
\end{proof}

It is known that both phase and norm retrieval are preserved when applying orthogonal projections to vectors. However, if the vectors do norm retrieval only for a subspace, this property may no longer hold. 

\begin{example}
	In $\mathbb{R}^3$, consider $$W=\spn\left\{u_1=\left(\dfrac{1}{\sqrt{2}}, 0, \dfrac{1}{\sqrt{2}}\right), u_2=\left(-\dfrac{1}{\sqrt{6}}, \dfrac{2}{\sqrt{6}}, \dfrac{1}{\sqrt{6}}\right)\right\}.$$
	Let $\{e_i\}_{i=1}^3$ be the canonical orthonormal basis for $\mathbb{R}^3$ and let $P$ be the orthogonal projection onto $\spn\{e_1, e_2\}$. Then
	\[Pu_1=\left(\dfrac{1}{\sqrt{2}}, 0, 0\right), \quad Pu_2=\left(-\dfrac{1}{\sqrt{6}}, \dfrac{2}{\sqrt{6}}, 0\right).\]
	Since $u_1$ and $u_2$ are orthogonal, they do norm retrieval on $W$. However, $Pu_1$ and $Pu_2$ are not orthogonal, so they do not do norm retrieval on $P(W)=\spn\{e_1, e_2\}$.
	
	Note that $P: W\to P(W)$ is invertible. Therefore, this example also shows that norm retrieval is not preserved under invertible operators.
\end{example}
The following important classification is from \cite{Ca}.
\begin{theorem}[\cite{Ca}]\label{T11}
	Let $\{P_i\}_{i=1}^M$ be projections onto subspaces $\{W_i\}_{i=1}^M$ of $\mathbb{R}^N$. The
	following are equivalent:
	
	\begin{enumerate}
		\item $\{P_i\}_{i=1}^M$ does norm retrieval.
		\item Given any orthonormal basis $\{\varphi_{ij}\}_{j=1}^{N_i}$ of $W_i$ and any subcollection $S\subset
		\{(i,j):1\le i\le M,1\le j\le N_i\}$ then,
		\[ \spn \{\varphi_{ij}\}^{\perp}_{(i,j)\in S} \perp \spn  \{\varphi_{ij}\}^{\perp}_{(i,j)\in S^{c}}.\]
		\item For any orthonormal basis $\{\varphi_{ij}\}_{i=1}^{N_i}$ of $W_i$, the vectors $\{\varphi_{ij}\}_{i=1, j=1}^{M,\ N_i}$ do norm retrieval.
	\end{enumerate}
\end{theorem}
It is necessary in (3) of Theorem \ref{T11} that this holds for every orthonormal basis of the $W_i$.

\begin{example}
	In $\mathbb{R}^3$ let $\{e_i\}_{i=1}^3$ be the canonical orthonormal basis and let $W_1=\spn\{e_1,e_3\}$ and $W_2=\spn\{e_2,e_3\}$.
	Let $x=(0,0,1)$ and $y=(1,1,0)$. Then
	\[ \|P_1x\|=\|P_1y\|=\|P_2x\|=\|P_2y\|=1.\]
	But $\|x\|=1$ and $\|y\|=\sqrt{2}\not= \|x\|$.
	However, the collection of vectors $\{e_1, e_3, e_2, e_3\}$ obviously does norm retrieval. 
\end{example}

In \cite{Ca}, the authors show that if $\{W_i\}_{i=1}^M$ does norm retrieval for $\mathbb{R}^N$, then $\sum_{i=1}^{M}\dim(W_i)\geq N$. We now show that if equality holds, then these subspaces must be pairwise orthogonal.

\begin{theorem}\label{pair_orth}
	Let $\{W_i\}_{i=1}^M$ be subspaces of $\mathbb{R}^N$ such that $\sum_{i=1}^{M}\dim(W_i)=N$. Then $\{W_i\}_{i=1}^M$ does norm retrieval if and only if these subspaces are pairwise orthogonal. 
\end{theorem}
\begin{proof}
	Let $\{\varphi_{ij}\}_{j=1}^{N_i}$ be any orthonormal basis of $W_i, i\in [M]$. If the subspaces $\{W_i\}_{i=1}^M$ are pairwise orthogonal, then the collection of all vectors $\{\varphi_{ij}\}_{i=1, j=1}^{M,\ N_i}$ forms an orthonormal basis for $\mathbb{R}^N$ and hence does norm retrieval. By Theorem \ref{T11}, $\{W_i\}_{i=1}^M$ does norm retrieval.
	
	Conversely, if $\{W_i\}_{i=1}^M$ does norm retrieval, then the collection of vectors $\{\varphi_{ij}\}_{i=1, j=1}^{M,\ N_i}$ does norm retrieval and hence spans $\mathbb{R}^N$. Since $\sum_{i=1}^{M}\dim(W_i)=N$, it follows that these vectors form a basis for $\mathbb{R}^N$. By Proposition \ref{Pro_span_ortho}, this basis must be orthonormal. Therefore, the subspaces $\{W_i\}_{i=1}^M$ are pairwise orthogonal.
\end{proof}

Regarding phase retrieval, we have the following theorem.

	\begin{theorem}
	Let $\{W_i\}_{i=1}^M$ be subspaces of $\mathbb{R}^N$ which do phase retrieval and assume
	\[ \sum_{i=1}^M\dim(W_i)=2N-1.\]
	Then $W_i\cap W_j=\{0\}$, for all $i\not= j$.
\end{theorem}
\begin{proof}
	If $x\in W_i\cap W_j$ is a unit vector, then we can choose orthonormal bases $\{\varphi_{ik}\}_{k=1}^{N_i}$ and $\{\varphi_{jk}\}_{k=1}^{N_j}$ for $W_i$ and $W_j$, respectively, such that $\varphi_{i1}=\varphi_{j1}=x$. Let $\{\varphi_{\ell k}\}_{k=1}^{N_\ell}$ be any orthonormal basis of $W_\ell$ for all $\ell\not = i, j$. By Theorem \ref{T21}, the collection of all vectors $\{\varphi_{\ell k}\}_{\ell=1, k=1}^{M,\ N_\ell}$ does phase retrieval. However, this collection contains at most $2N-2$ vectors, which contradicts the fact that at least $2N-1$ vectors are required for phase retrieval in $\mathbb{R}^N$.
\end{proof}
\begin{theorem} Let $\{W_i\}_{i=1}^M$ be subspaces of $\mathbb{R}^N$ allowing phase retrieval (respectively, norm retrieval). If $W_i=U_i\oplus V_i$ for all
	$i\in [M]$, then the family of subspaces $\{U_i\}_{i=1}^M \cup \{V_i\}_{i=1}^M$ does phase retrieval (respectively, norm retrieval) in $\mathbb{R}^N$.
\end{theorem}

\begin{proof} Let $P_i, Q_i$, and $R_i$ be the orthogonal projections onto $W_i, U_i$, and $V_i$, respectively, for $i \in [M]$. 
	By assumption, $Q_ix\perp R_ix$ for all $i\in [M]$. Hence, if
	\[ \|Q_ix\|=\|Q_iy\|\mbox{ and } \|R_ix\|=\|R_iy\| \mbox{ for all } i\in [M],\]
	then
	\[ \|P_ix\|^2=\|Q_ix+R_ix\|^2= \|Q_ix\|^2+\|R_ix\|^2=\|Q_iy\|^2+\|R_iy\|^2 =\|P_iy\|^2.\]
	The result follows immediately.
\end{proof}
\begin{remark}
	The converse of the theorem above is false. For example, let 
	$$W_1=\spn\{e_1\}, W_2=\spn\{e_2\}, W_3=\spn\{e_3\}, W_4=\spn\{e_3\},$$ where $\{e_i\}_{i=1}^3$ is the canonical orthonormal basis for $\mathbb{R}^3$.
	Then $\{W_i\}_{i=1}^4$ does norm retrieval but $\{W_1\oplus W_3, W_2\oplus W_4\}$ fails norm retrieval.
\end{remark}

The following theorem completely classifies when two proper subspaces perform norm retrieval.

\begin{theorem}
	Two proper subspaces of $\mathbb{R}^N$ do norm retrieval if and only if they are the orthogonal complements of each other. 
\end{theorem}
\begin{proof}
	Suppose that $W_1, W_2$ are proper subspaces of $\mathbb{R}^N$ with corresponding orthogonal projections $P_1, P_2$. 
	If they are orthogonal complements of each other, then it is clear that they do norm retrieval.
	
	Now we suppose that $\{W_1, W_2\}$ does norm retrieval. Then we must have $\dim(W_1)+\dim(W_2)\geq N$. Assume, toward a contradiction, that $\dim(W_1)+\dim(W_2)=p+q>N$. Then $k:=\dim(W_1\cap W_2)>0$. Let $\{e_1, e_2, \ldots, e_k\}$ be an orthonormal basis for $W_1\cap W_2$, and extend it to orthonormal bases $\{e_1, \ldots, e_k, u_{k+1}, \ldots, u_p\}$ for $W_1$ and $\{e_1, \ldots, e_k, v_{k+1}, \ldots, v_q\}$ for $W_2$.
	
	Since $\{W_1, W_2\}$ does norm retrieval, it follows that the set of vectors $\{e_1, \ldots, e_k, u_{k+1}, \ldots, u_p, v_{k+1}, \ldots, v_q\}$ does norm retrieval in $\mathbb{R}^N$. Therefore, by Proposition \ref{Pro_span_ortho}, this set must be an orthonormal basis for $\mathbb{R}^N$. 
	
	Let 
	\[x=2e_1+u_{k+1}+v_{k+1} \mbox{ and } y = e_1+2u_{k+1}+2v_{k+1}.\] Then we have 
	\[\|P_1x\|=\|2e_1+u_{k+1}\|=\sqrt{5}=\|e_1+2u_{k+1}\|=\|P_1y\|, \] and 
	\[\|P_2x\|=\|2e_1+v_{k+1}\|=\sqrt{5}=\|e_1+2v_{k+1}\|=\|P_2y\|.\] But 
	\[\|x\|=\sqrt{6}\not=\|y\|=3.\] This contradicts the fact that $\{W_1, W_2\}$ does norm retrieval.
	
	Thus, we have shown that $\dim(W_1)+\dim(W_2)=N$. The conclusion then follows by Theorem \ref{pair_orth}. 
\end{proof}

In \cite{A}, the authors prove that the collection of norm-retrievable frames in $\mathbb{R}^N$ is not dense in the set of all $M$-element frames, where $M < 2N-1$.

\begin{theorem}[\cite{A}]
	Let $\{\varphi_i\}_{i=1}^M$ be a frame for $\mathbb{R}^N$ with $M< 2N-1$ which fails norm retrieval. Then there exists an $\epsilon>0$ such that whenever
	$\{\psi_i\}_{i=1}^M$ are vectors satisfying
	\[ \sum_{i=1}^M\|\psi_i-\varphi_i\|< \epsilon,\]
	then $\{\psi_i\}_{i=1}^M$ also fails norm retrieval.
	
\end{theorem}

\begin{remark}
	The theorem fails for $M\geq2N-1$ since the full spark frames are dense in the set of all frames, and every full spark frame with $M\ge 2N-1$ elements does phase retrieval, and hence norm retrieval.
\end{remark}

We now extend this theorem to subspaces. To this end, we need some lemmas.

\begin{lemma}\label{L7}
	If $\|x\|=1$ and $\|y-x\|<\epsilon$, then
	\[\left \|\frac{y}{\|y\|}-x\right \|< 2\epsilon.
	\]
\end{lemma}

\begin{proof}
	We compute:
	\[ 1-\epsilon < \|x\|-\|y-x\|\le \|y\|\le \|x\|+\|y-x\|<1+\epsilon.\]
	Now
	\begin{align*}
		\left \|\frac{y}{\|y\|}-x\right \|&\le \left \|\frac{y}{\|y\|}-y\right \|+\|y-x\|\\
		&<\frac{1}{\|y\|}
		\|y-(\|y\|)y\|+\epsilon=|1-\|y\||+\epsilon< \epsilon+\epsilon=2\epsilon,
	\end{align*} which is the claim.
\end{proof}

\begin{lemma}\label{GrS}
	Assume $\{x_i\}_{i=1}^p$ is an orthonormal basis for a subspace $W$ of $\mathbb{R}^N$. Given $\epsilon>0$, there exists a
	$\delta>0$ such that if $\{y_i\}_{i=1}^p$ is a set of vectors satisfying $\|y_i-x_i\|<\delta$ for all $i\in [p]$, then there exists an orthonormal basis $\{z_i\}_{i=1}^p$ for $\spn\{y_i\}_{i=1}^p$
	with
	\[ \|z_i-x_i\|< \epsilon, \mbox{ for all } i\in [p].\]
\end{lemma}
\begin{proof} We can choose $\delta>0$ small enough so that $\{y_i\}_{i=1}^p$ is linearly independent. Let  $\{u_i\}_{i=1}^p$ be the orthogonal set of vectors obtained by applying the Gram-Schmidt orthogonalization process to $\{y_i\}_{i=1}^p$ and let $z_i=\dfrac{u_i}{\|u_i\|}, i\in [p]$. Then, 
	\[u_1=y_1 \mbox{ and } u_k = y_k-\sum_{i=1}^{k-1}\langle y_k, z_i\rangle z_i, \mbox{ for } k = 2, \ldots, p.\] 
	It follows that,
	for all $k\geq2$, we have			
	\begin{align*}
		\|u_k-y_k\|&=\|\sum_{i=1}^{k-1}\langle y_k, z_i\rangle z_i\|\\
		&\leq \sum_{i=1}^{k-1}|\langle y_k, z_i\rangle|\\
		&= \sum_{i=1}^{k-1}|\langle y_k - x_k, z_i\rangle +\langle z_i-x_i, x_k\rangle +\langle x_i, x_k\rangle |\\
		&\leq (k-1)\|y_k-x_k\|+\sum_{i=1}^{k-1}\|z_i-x_i\|.
	\end{align*}
	Therefore, for all $k\geq2$, we have
	\[\|u_k-x_k\|\leq \|u_k-y_k\|+\|y_k-x_k\|\leq k\|y_k-x_k\|+\sum_{i=1}^{k-1}\|z_i-x_i\|.\]
	Moreover, by Lemma \ref{L7}, if $\|y_1-x_1\|<\delta$, then $\|z_1-x_1\|<2\delta$. It follows that there exists a constant $C>0$ such that 
	\[\|z_i-x_i\|< C\delta, \mbox{ for all } i\in [p].\] The conclusion then follows.
\end{proof}
\begin{lemma}\label{lem_dim} Let $P$ and $Q$ be orthogonal projections on $\mathbb{R}^N$. If $\|P-Q\|<1$, then $\dim(P(\mathbb{R}^N))=\dim(Q(\mathbb{R}^N))$.
\end{lemma}

\begin{proof}
	Let $\{u_i\}_{i=1}^p$ be an orthonormal basis for $P(\mathbb{R}^N)$. We claim that $\{Qu_i\}_{i=1}^p$
	is linearly independent. If not, there exist scalars $\{a_i\}_{i=1}^p$, not all zero, such that
	$\sum_{i=1}^p a_iQu_i=0.$
	Then
	\begin{align*}
		\|\sum_{i=1}^pa_iu_i\|&\le \|\sum_{i=1}^pa_iQu_i\|+\|\sum_{i=1}^pa_iPu_i-\sum_{i=1}^pa_iQu_i\|\\
		&\leq \|P-Q\|\|\sum_{i=1}^pa_iu_i\|<\|\sum_{i=1}^pa_iu_i\|,
	\end{align*}
	which is a contradiction.
	
	It follows that $\dim(Q(\mathbb{R}^N))\geq\dim(P(\mathbb{R}^N))$. Reversing the roles of $P$ and $Q$ proves $\dim(P(\mathbb{R}^N))=\dim(Q(\mathbb{R}^N))$.
\end{proof}
\begin{theorem}\label{T9}
	Let \(\{W_i\}_{i=1}^M\) be subspaces of \(\mathbb{R}^N\) with corresponding orthogonal projections \(\{P_i\}_{i=1}^M\), and suppose that 
	$\sum_{i=1}^M \dim(W_i) < 2N - 1.$
	If \(\{P_i\}_{i=1}^M\) fails norm retrieval, then there exists \(\delta > 0\) such that for any set of orthogonal projections \(\{Q_i\}_{i=1}^M\) satisfying \(\|Q_i - P_i\| < \delta\) for all \(i \in [M]\), the set \(\{Q_i\}_{i=1}^M\) also fails norm retrieval.
\end{theorem}
\begin{proof}
	By Theorem \ref{T11}, for each $i\in [M]$, there exists an orthonormal basis $\{\varphi_{ij}\}_{j=1}^{N_i}$ for $W_i$ such that the combined collection $\{\varphi_{ij}\}_{i=1, j=1}^{M,\ N_i}$ fails norm retrieval.
	Let $\epsilon>0$ be such that whenever $\{\psi_{ij}\}_{i=1, j=1}^{M,\ N_i}$ are vectors satisfying  $\|\psi_{ij}-\varphi_{ij}\|<\epsilon$, then $\{\psi_{ij}\}_{i=1, j=1}^{M,\ N_i}$ fails norm retrieval. Let $\delta>0$ be as in Lemma \ref{GrS}, and
	assume that $\{Q_i\}_{i=1}^M$ is a family of orthogonal projections on $\mathbb{R}^N$ such that $\|Q_i-P_i\|<\delta$ for all $i\in [M]$. Note that by Lemma \ref{lem_dim}, $\dim(Q_i(\mathbb{R}^N))=\dim(P_i(\mathbb{R}^N))$ for all $i\in [M]$. Moreover,
	\[ \|Q_i\varphi_{ij}-\varphi_{ij}\|=\|Q_i\varphi_{ij}-P_i\varphi_{ij}\|<\delta,\mbox{ for all } i, j.\] By Lemma \ref{GrS}, for each $i\in [M]$, there exists an orthonormal basis $\{\psi_{ij}\}_{j=1}^{N_i}$ for $Q_i(\mathbb{R}^N)$ with $\|\psi_{ij}-\varphi_{ij}\|<\epsilon$. Therefore, $\{\psi_{ij}\}_{i=1, j=1}^{M,\ N_i}$ fails norm retrieval and so $\{Q_i\}_{i=1}^M$ also fails norm retrieval.
	\end{proof}

\vspace{0.1in}
{\bf Acknowledgments}
\vspace{0.1in}

The first author was supported by Hue University under the
grant number DHH2025-03-205.

\end{document}